\documentclass[12pt]{article}

\usepackage{textcomp}
\usepackage{amssymb}
\usepackage{bbm}
\usepackage{fancyhdr}
\usepackage{amsmath}
\usepackage{amsbsy,amsthm}
\usepackage{amscd}
\usepackage{latexsym}
\usepackage{graphicx}   
\usepackage{pdfsync}
\usepackage{blkarray}
\usepackage{multirow}
\usepackage{hyperref}
\usepackage{color}
\usepackage{comment}
\usepackage{framed}

\newcommand{\N}{\mathbb{N}}
\newcommand{\Z}{\mathbb{Z}}
\newcommand{\E}{\mathbb{E}}

\newcommand{\Q}{\mathbb{Q}}

\newtheorem{lemma}{Lemma}
\newtheorem{theorem}{Theorem}

\numberwithin{equation}{section}

\newcommand{\C}{\mathcal{C}}

\renewcommand{\E}{\mathbb{E}}
\newcommand{\Pp}{\mathbb{P}}

\newcommand{\F}{{\mathcal F}}

\begin{document}

\title{Anisotropic oriented percolation  in high dimensions}

\author{
Pablo Almeida Gomes \thanks{ICEx,  University of Minas Gerais, Minas Gerais, Brazil, pabloag@ufmg.br}
\and
Alan Pereira \thanks{ICEx,  University of Minas Gerais, Minas Gerais, Brazil, alanand@ufmg.br}
\and
Remy Sanchis \thanks{ICEx,  University of Minas Gerais, Minas Gerais, Brazil, rsanchis@mat.ufmg.br}}

\date{}

\maketitle

\begin{abstract}
In this paper we study anisotropic oriented percolation on $\Z^d$ for $d\geq 4$ and show that the local condition for  phase transition is closely related to the mean-field condition. More precisely, we show that if the sum of  the local probabilities is strictly greater than one and each probability is not too large, then percolation occurs.


\end{abstract}

\section{Introduction}

A common feature of lattice systems is that they approach, in some sense, a mean-field behavior as the dimension of the lattice grows. 

In the particular case of oriented percolation, the seminal paper by Cox and Durrett \cite{CD} shows, among other important results, that the asymptotic behavior of the critical parameter is $1/d$. A little earlier, Holley and Liggett \cite{HL} proved the occurrence  of an analogous behavior for the high dimensional contact process critical rate and, recently, a similar result was proven for the contact process with random rates on a high dimensional percolation open cluster, see \cite{XX}.  For the non-oriented case, Kesten \cite{K} and Gordon \cite{G} independently showed that the critical parameter is asymptotically $1/2d$ and, in the last three decades, a rather complete mean-field picture of high dimensional non-oriented percolation has emerged; see \cite{R} and references therein. 

For $d\geq 3$,  little is understood about the phase diagram of anisotropic percolation. The results of \cite{CD,K,G} are statements about one point in the parameter domain, the isotropic point, whereas in \cite{CLS} the authors show that the critical surface is everywhere continuous for a particular setup in the non-oriented case. The behavior on the border of the domain, related to the so-called dimensional crossover phenomenon, is  studied in \cite{SS}, but still with quite modest results.

In this paper we analyze the interior of the domain in a region containing the isotropic point, showing that the critical surface behaves nicely around this point and that the anisotropy introduced in the system does not create any unexpected behavior.

\subsection{The model}

We will consider the anisotropic oriented edge percolation model in $\Z^d$. Let  $\{e_1, \dots, e_d\}$ be the set of canonical unit vectors of $\Z^d$. Given $0< p_1, \ldots, p_d <1$, we declare each edge $\langle x, x+e_i \rangle$ to be open independently of each other with probability $p_i$ for $i=1,\dots,d$. We will denote the corresponding probability measure simply by $\Pp$.

An {\it oriented path} of length $n$ starting at the origin  in $\Z^d$ is a path $(x_0,x_1,\dots, x_n)$ such that $x_0=0$ and $x_i-x_{i-1}\in \{e_1, \ldots, e_d\}$ for $i=1,\dots, n$. Let $C_0$ be the open cluster of the origin, that is, the set of vertices $x \in \Z^d$ such that there is an open path from $0$ to $x$; we let $|C_0|$ denote the size of $C_0$.

We are now ready to state our main theorem.




\begin{theorem}\label{main}
Let $\varepsilon > 0$. For $d \geq 4$, let $p_1, \ldots, p_d$ be non-negative numbers such that

\begin{enumerate}
\item $p_1+\cdots+ p_d \geq 1+\varepsilon$,

\item $\underset{1\leq i\leq d}{\max}\left\{ \displaystyle\frac{p_i}{p_1 + \cdots + p_d} \right\} < \left \lceil \dfrac{10}{\varepsilon}  \right \rceil^{-1}$,
\end{enumerate}
then
\[\Pp\big\{|C_0| =\infty \big\}>0.\]
\end{theorem}

{\bf Remark:} Note that if $p_1+\cdots+ p_d < 1$ then a straightforward branching process argument shows that $\Pp\big\{|C_0| =\infty \big\}=0$. Note also that, although the result is non-asymptotic, it only makes sense for $d$ of order $1/\varepsilon$. We will see in the course of the proof of Theorem \ref{main} (more precisely on the remark just bellow  \eqref{equation:cotastirling}) that, for any $C<1,$ the upper bound of Condition 2 above could be taken as $C\varepsilon$ , as long as the dimension is taken big enough. We  also mention that, for the isotropic case,  Theorem \ref{main} gives the bound $p_c(\mathbb{Z}^d) \leq 1/d + 10/d^2$.

The strategy of the proof is similar to the one in  Cox-Durrett  \cite{CD}: we build a martingale and prove the convergence to a positive limit showing that it has bounded second moment. The main difficulty here is to estimate the second moments of the martingale. We do this by converting the martingale problem into a random walk problem and comparing the asymmetric case with the symmetric one.

\subsection{Proof of main result}

In this section we prove the main result modulo a lemma which is stated in the course of the proof.

\begin{proof} [Theorem \ref{main}] 
We want to prove that the open cluster of the origin is infinite, which happens if and only if there exists an infinite  open path starting at the origin. Thus our problem can be naturally converted to a counting of the number of open paths starting on $0$, as long as we have a good control of the second moment of this random variable.

For each $n \in \N$ let $V_n$ be the set of the vertices in the $n$-th level , i.e.,
\begin{equation}
V_n = \{x \in \Z^d: x_1 + \cdots + x_d = n, \: x_i \geq 0 \}, 
\end{equation} 
and $\mathcal{C}_n$ to be the set of all possible oriented paths from the origin to $V_n$, i.e.,
\begin{equation}
\mathcal{C}_n = \{ \gamma = (0,v_1,v_2, \ldots, v_n) : v_j - v_{j-1} \in V_1; \ \ j=1, \ldots, n\} 
\end{equation}

We define $X_n$ to be the random variable which counts the open paths from the origin up to level $n$, i.e.,
\begin{equation} \label{Def:Xn}
X_n := \sum_{\gamma \in \C_n}  1_{\{\gamma \ \ \textrm{is open}\}},
\end{equation} 
and write $\mu := \E[X_1] = p_1+ \cdots + p_d$. A simple calculation shows that $\E[X_n] = \mu^n$. Now, define
\begin{equation} \label{Def:Wn}
W_n := \frac{X_n}{\mu^n}.
\end{equation}

We observe that $\{ W_n \}_{n \in \N}$ is a positive martingale. For this, consider $x \in V_n$ and let $\C_x = \{ \gamma \in \C_n: f(\gamma)=x\}$, where $f(\gamma)$ denotes the final vertex of $\gamma$. Define also the random variables
\[ Y_x = \sum_{i=1}^d 1_{\{\langle x, x+e_i \rangle \ \ \textrm{is open}\}} \ \ \mbox{and} \ \  N_x = \sum_{\gamma \in \C_x} 1_{ \{ \gamma \ \ \textrm{is open}\}},\]
where $Y_n$ counts the number of oriented open edges leaving $x$ and  
$N_x$ counts the number of oriented open paths from 0 to $x$, respectively.
Observe that
\begin{equation} \label{equation:XnNx}
X_n = \sum_{x \in V_n} N_x  \ \ \mbox{and} \ \ X_{n+1} = \sum_{x \in V_n} N_x \cdot Y_x.
\end{equation}
Let $\F_n = \sigma(X_1, \ldots, X_n)$; observe that $Y_x$ is independent of $\F_n$ for each $x \in V_n$ and that $N_x$ is $\F_n$-measurable. We also have $\E[Y_x] =\mu$ for all $x$, so
\begin{align*} 
\E[X_{n+1} |\F_n] &= \sum_{x \in V_n}\E[ N_x \cdot Y_x| \F_n] = \mu X_n.
\end{align*}
 Thus $\E[W_{n+1} |\F_n] = W_n$, as wanted.
 
 Since $W_n$ is a positive martingale, it converges to a non-negative random variable $W$. Assume for the moment the following lemma.

\begin{lemma}  \label{lemma:main}
Let $\{W_n\}_n$ be as defined in (\ref{Def:Wn}). Then, under the conditions of Theorem \ref{main} we have
\[ \sup_n \E[W_n^2]<\infty. \]
\end{lemma}

From this lemma it follows that $W_n$ converges to $W$  in $L_1$ and since $\E(W_n)=1$ we have $\Pp(W>0)>0$. Noticing that $W>0$ implies $X_n>0$ for all $n$, the theorem follows. 
\end{proof}

 The proof of Lemma \ref{lemma:main} in the anisotropic case requires more than a direct adaptation of Cox-Durrett results. The next sections are dedicated to this work. 
 
 The structure of the remainder of the paper is the following: in Section \ref{preliminary} we convert the martingale problem of Lemma \ref{lemma:main} into a random walk problem, in Section \ref{rw} we compare asymmetric random walks with the symmetric case to finish the proof of Lemma \ref{lemma:main} and in Section \ref{final} we make some final remarks.

 \section{Preliminary Lemmas}\label{preliminary}

In this section we state and prove four lemmas with the goal of converting the martingale problem of Lemma \ref{lemma:main} into a random walk problem.

\subsection{Open paths equivalencies}

In the first lemma we give a criterion for bounding $\sup_m \E[W_m^2]$. To do that, we introduce the following notation. Given a path $\gamma=(0,v_1, \ldots, v_n) \in \C_n$, we define, for each $i \in \{1, \ldots, n\}$,
\begin{equation}
i(\gamma) := v_i \ \ \mbox{and} \ \ f(\gamma):=v_n.
\end{equation}

\begin{lemma} \label{an-iff-Wn}
Let
\begin{equation} \label{Def:an}
a_n := \frac{1}{\mu^{2n}} \sum_{(\gamma_1, \gamma_2) \in \C_n^2, f(\gamma_1)=f(\gamma_2)} \Pp(\gamma_1, \gamma_2 \ \ \mbox{open}).
\end{equation}
Then
\[ \sup_m \E[W_m^2] < \infty\; \text{ iff }\; \sum_{n=1}^{\infty} a_n < \infty.\]
\end{lemma}

\begin{proof}
By (\ref{equation:XnNx}) we have
\begin{align*}
\E[X_{n+1}^2| \F_n] = \E \left[  \left. \sum_{(x,y) \in V_n^2} N_x Y_x N_y Y_y \right| \F_n \right] = \mu^2 X_n^2 + ({\text Var} {Y_0}) \sum_{ x \in V_n} N_x^2.
\end{align*}
 Therefore
\begin{align*}
\E[W_{n+1}^2] &= \frac{\mu^2 \E[X_n^2] }{\mu^{2n+2}} + \frac{({\text Var} {Y_0})}{\mu^2} \frac{\E \left[\sum_{ x \in V_n} N_x^2 \right] }{\mu^{2n}}.
\end{align*}
Using the definition of $N_x$ we can see that $\E[N_x^2 ] = \displaystyle \sum_{(\gamma_1, \gamma_2) \in \C_x^2} \Pp(\gamma_1, \gamma_2 \ \ \mbox{open})$
so
\begin{align*}
\frac{\E \left[\sum_{ x \in V_n} N_x^2 \right] }{\mu^{2n}} &= \frac{\sum_{ x \in V_n} \sum_{(\gamma_1, \gamma_2) \in \C_x^2} \Pp(\gamma_1, \gamma_2 \ \ \mbox{open})  }{\mu^{2n}} \\
&= \frac{1}{\mu^{2n}} \sum_{(\gamma_1, \gamma_2) \in \C_n^2, f(\gamma_1)=f(\gamma_2)} \Pp(\gamma_1, \gamma_2 \ \ \mbox{open}),
\end{align*}
which is exactly the definition of $a_n$. Hence
\begin{align*}
 \E[W_{n+1}^2] = \E[W_n^2] + \frac{({\text Var} {Y_0}) }{\mu^2} a_n.
\end{align*}
Iterating the recursion above, we have
\begin{equation*}
\E[W_{n+1}^2] = \E[W_1^2] + \frac{({\text Var} {Y_0}) }{\mu^2} \sum_{j=1}^n a_j,
\end{equation*}
and the result follows.
\end{proof}

Given $m \in \N$, let
\begin{equation} \label{Def:An}
A_m := \{(\gamma_1,\gamma_2) \in \C_m^2: i(\gamma_1) \neq i(\gamma_2) \ \ \text{ for all } i \neq m \ \ \mbox{and}  \ \ f(\gamma_1)= f(\gamma_2) \}
\end{equation}
be the set of the pair of paths of size $m$ which meet on the their final vertices and define
\begin{equation} \label{Def:bn}
b_m := \sum_{(\gamma_1, \gamma_2) \in A_m}  \frac{\Pp(\gamma_1, \gamma_2 \ \ \mbox{open})}{\mu^{2m}} .
\end{equation}

\begin{lemma}
Let $a_n$ as defined in (\ref{Def:an}) and $b_m$ as defined in (\ref{Def:bn}). Then
\begin{equation}
\sum_{n=1}^{\infty} a_n = \sum_{j=1}^{\infty}\left( \sum_{m=1}^{\infty} b_m \right)^j.
\end{equation}
\end{lemma}

\begin{proof}
Recall that
\begin{equation} \label{equation:an}
a_n := \frac{1}{\mu^{2n}} \sum_{(\gamma_1, \gamma_2) \in \C_n^2, f(\gamma_1)=f(\gamma_2)} \Pp(\gamma_1, \gamma_2 \ \ \mbox{open}).
\end{equation}
We will show that the set $\{(\gamma_1, \gamma_2) \in \C_n^2: f(\gamma_1) = f(\gamma_2)\}$ can be partitioned by the number of vertices in the intersection of $\gamma_1$ and $\gamma_2$.
Let
\[
I(\gamma_1, \gamma_2) = \{i:i(\gamma_1) = i(\gamma_2) \}  \]
and setting $M = m_1+\cdots+ m_j$, let
\[C(m_1, \dots,m_j) = \big\{(\gamma_1, \gamma_2) \in \C_M^2: I(\gamma_1, \gamma_2)=\{m_1,m_1+m_2, \dots, m_1+\cdots+ m_j\} \big \},
\]
then
\begin{equation}
\{(\gamma_1, \gamma_2) \in \C_M^2: f(\gamma_1) = f(\gamma_2)\} =\bigsqcup_{j=1}^n\bigsqcup_{\substack{(m_1, \ldots, m_j) \in \N^j \\  m_1+ \cdots + m_j = M}}  C(m_1, \ldots,m_j).\label{cup}
\end{equation}

Given two paths $\gamma_1= (0,v_1, \ldots, v_n)$ and $\gamma_2=(0,w_1, \ldots, w_m)$  we define the \textit{concatenation} of $\gamma_1$ and $\gamma_2$ by $\gamma_1 \circ \gamma_2 :=(0,v_1, \ldots, v_n, v_n+w_1, \ldots, v_n+w_m)$.
Observe that, given a sequence of positive integers $(m_1, \ldots, m_j) \in \N^j$, and recalling \eqref{Def:An}, we have
\begin{align*}
C(m_1, \ldots, m_j) = \{ (\gamma_1, \gamma_2) \in \C_{M}^2 : \gamma_1 = \gamma_{1,1} \circ \cdots &\circ \gamma_{1,j}, \gamma_2 = \gamma_{2,1} \circ \cdots \circ \gamma_{2,j}, \\
& \ \ (\gamma_{1,k}, \gamma_{2,k}) \in A_{m_k}, \forall k=1, \ldots, j \}.
\end{align*}
Using (\ref{cup}) we can rewrite (\ref{equation:an}) as
\begin{equation} \label{equation:lemma-an-bn}
\sum_{n=1}^{\infty} a_n = \sum_{j=1}^{\infty} \sum_{ (m_1, \ldots, m_j) \in \N^j} \ \ \sum_{(\gamma_1, \gamma_2) \in C(m_1, \cdots, m_j)} \frac{\Pp(\gamma_1, \gamma_2 \ \ \mbox{open})}{\mu^{2M}} .
\end{equation}
From the definitions of $C(m_1, \ldots, m_j)$ and $b_m$ mentioned earlier, it follows that
\begin{align*}
\sum_{(\gamma_1, \gamma_2) \in C(m_1, \cdots, m_j)} \frac{\Pp(\gamma_1, \gamma_2 \ \ \mbox{open})}{\mu^{2M}} &= \sum_{  \substack{(\gamma_{1,1} \circ \cdots \circ \gamma_{1,j}, \gamma_{2,1} \circ \cdots \circ \gamma_{2,j}) \\ (\gamma_{1,k},\gamma_{2,k}) \in A_{m_k}, k=1, \ldots, j} } \prod_{k=1}^{j} \frac{\Pp(\gamma_{1,k}, \gamma_{2,k} \ \ \mbox{open})}{\mu^{2m_k}} \\
&=  \prod_{k=1}^{j} \left[ \sum_{(\gamma_1, \gamma_2) \in A_{m_k}} \frac{\Pp(\gamma_{1}, \gamma_{2} \ \ \mbox{open})}{\mu^{2m_k}}\right] \\
&= \prod_{k=1}^j b_{m_k}.
\end{align*}

Substituting the expression above in  (\ref{equation:lemma-an-bn}) we obtain 
\begin{align*}
\sum_{n=1}^{\infty} a_n &= \sum_{j=1}^{\infty} \sum_{(m_1, \cdots, m_j) \in \N^j}  \left[ \prod_{k=1}^j b_{m_k}\right] = \sum_{j=1}^{\infty}\left( \sum_{m=1}^{\infty} b_m \right)^j,
\end{align*}
and the result follows.
\end{proof}

\subsection{Converting the problem to a random walk}

Using the lemmas in the previous section, we have, so far,  
\[\displaystyle \sup_m \E[W_m^2]<\infty \;\text{ iff }\;  \sum_{m=1}^{\infty} b_m <1;\]
we will now use random walks to compute $b_n$. In the remainder of the text, we will consider several independent random walks and we use $\Q$ to denote the probability on a space where they all live in harmony.

We say that $q = (q_1, \ldots, q_d)$ is a ($d$-dimensional) \textit{positive vector} if $0 \leq q_1, \ldots, q_d $ and $q_1 + \cdots + q_d >0$, and we say that $q$ is a ($d$-dimensional) \textit{probability vector} if it is a positive vector with $q_1 + \cdots + q_d = 1$. Given a positive vector $q$ we say that $\{S_n\}_n$ is the oriented random walk \textit{associated to $q$}  if
\begin{equation} \label{Def:Sn}
S_n = \xi_1+ \cdots + \xi_n,
\end{equation}
where $\xi_1, \ldots, \xi_n$ are i.i.d. random variables with
\[ \Q(\xi_1 = e_i)=  \frac{q_i}{q_1+\cdots+q_d}, \ \ \mbox{ for all $i=1, \ldots, d$}.\]

Given two independent random walks associated to $q$, $\{S^1_n\}_n$ and $\{S^2_n\}_n$ we define
\begin{equation} \label{Def:tau}
\tau = \tau(\{S^1_n\}_n, \{S^2_n\}_n): = \inf\{n \geq 1: S^1_n=S^2_n\}.
\end{equation}

\begin{lemma}
Let $b_m$ be as defined in (\ref{Def:bn}), $A_m$ be as in (\ref{Def:An}) and $\{S^1_n\}_n$, $\{S^2_n\}_n$ be two independent random walks associated to  $p=(p_1, \ldots, p_d)$. Then

\[\sum_{m=1}^{\infty} b_m < 1 \;\text{ iff }\; \sum_{m=2}^{\infty} \Q(\tau = m) < 1 - \frac{1}{\mu}.\]
\end{lemma}

\begin{proof}
Observe that
\begin{align*}
b_1 &= \sum_{(\gamma_1, \gamma_2) \in A_1} \frac{\Pp(\gamma_{1}, \gamma_{2} \ \ \mbox{open})}{\mu^{2}} 
= \sum_{i=1}^d  \frac{\Pp(\langle 0,e_i \rangle, \langle 0,e_i \rangle \ \ \mbox{open})}{\mu^{2}} 
= \sum_{i=1}^d \frac{p_i}{\mu^2} 
= \frac{1}{\mu}.
\end{align*}
Also, for $m\geq 2$,  
we have
\begin{align*}
b_m &= \sum_{(\gamma_1, \gamma_2) \in A_m} \frac{\Pp(\gamma_{1}, \gamma_{2} \ \ \mbox{open})}{\mu^{2m}}  \\
&= \sum_{(\gamma_1, \gamma_2) \in A_m}   \Q\big((0,S^1_1, \cdots, S^1_m \big)= \gamma_1) \cdot \Q \big((0,S^2_1, \cdots, S^2_m) = \gamma_2\big) \\
&= \Q(\tau =m),
\end{align*}
thus
\[\sum_{m=1}^{\infty} b_m = \frac{1}{\mu} + \sum_{m=2}^{\infty} \Q(\tau = m),\]
and the result follows.
\end{proof}

\begin{lemma}
Let $\{S^1_n\}_n$ and $\{S^2_n\}_n$ be two independent random walks associated to a positive vector $q$ and let $\tau$ be as defined in (\ref{Def:tau}). Then
\[ \sum_{m=1}^{\infty} \Q(\tau=m) \leq 1-\frac{1}{\mu} \;\mbox{ iff }\;  \sum_{k=1}^{\infty} \Q(S^1_{k}=S^2_{k}) \leq  \mu -1. \]
\end{lemma}

\begin{proof}
Observe that
\begin{equation}
\Q(S^1_m=S^2_m) = \sum_{n \leq m} \Q(\tau=n) \Q(S^1_{m-n}=S^2_{m-n}),
\end{equation}
so
\begin{align*}
\sum_{m=1}^{\infty} \Q(S^1_m=S^2_m) &= \sum_{m=1}^{\infty} \sum_{n \leq m} \Q(\tau=n) \Q(S^1_{m-n}=S^2_{m-n}) \\
&= \sum_{n=1}^{\infty} \left[  \Q(\tau=n) \sum_{m \geq 0} \Q(S^1_{m}=S^2_{m}) \right] \\
&= \left( \sum_{n=1}^{\infty}  \Q(\tau=n)  \right) \left( 1+ \sum_{m=1}^{\infty} \Q(S^1_{m}=S^2_{m}) \right),
\end{align*}
and therefore
\begin{equation}
\sum_{n=1}^{\infty}  \Q(\tau=n) = \dfrac{\sum_{m=1}^{\infty} \Q(S^1_m=S^2_m) }{1+ \sum_{m=1}^{\infty} \Q(S^1_m=S^2_m) } = 1 - \dfrac{1}{1+ \sum_{m=1}^{\infty} \Q(S^1_m=S^2_m)}.
\end{equation}
Finally, we have
\[ 1 - \dfrac{1}{1+ \sum_{m=1}^{\infty} \Q(S^1_m=S^2_m)} \leq 1 - \frac{1}{\mu} \;\text{ iff }\; \sum_{m=1}^{\infty} \Q(S^1_m=S^2_m) \leq \mu -1,\]
and this finishes the proof.
\end{proof}

\section{Analysis of the random walks}\label{rw}

In this section, we will estimate the maximal probability of two i.i.d random walks meeting in a fixed time as a function of their parameters.

Given a probability vector $q=(q_1, \ldots, q_d)$, let $\{S^1_n\}_n$ and $\{S^2_n\}$ be two independent random walks associated to $q$ and define

\begin{equation} \label{def:sigmaq}
\lambda(q) := \sum_{n=1}^{\infty} \Q(S^1_{n}=S^2_{n}).
\end{equation}

 Combining the results of the previous section, and observing that $\Q(\tau = 1) > 0$, we have
\begin{equation} \label{eq:main}
\text{If }\; \lambda(q) \leq \mu -1, \;\text{ then }\;\sup_{n} \E[W_n^2]<\infty.
\end{equation}

In this section, we investigate the behavior of $\lambda(q)$.

\begin{theorem} \label{teo:meeting}
Let $d \geq 4$ and $4 \leq m \leq d$ be integers. Consider the $d$-dimensional probability vector $q^* = q^*_d(m) = (1/m, 1/m, \dots, 1/m, 0, \dots, 0)$, and let $\lambda(q)$ be as  in (\ref{def:sigmaq}). Then 
\begin{enumerate}
\item[(a)] $\lambda(q^*) \leq \dfrac{10}{m}$

\item[(b)] for all $d$-dimensional probability vectors  $q$ such that $\max_i{q_i} \leq \dfrac{1}{m}$ we have 
\[ \lambda(q) \leq \lambda(q^*).\]
\end{enumerate}
\end{theorem}

We will prove Theorem \ref{teo:meeting} in the next sections, but before that we use it to prove Lemma \ref{lemma:main}.

\begin{proof}[Lemma \ref{lemma:main}]
Let $q_i = \frac{p_i}{\sum p_i}$. Under the hypothesis of the Theorem \ref{main}, we have $q_i \leq \left \lceil \varepsilon/10 \right \rceil^{-1}$ for  $1\leq i\leq d$; taking $m = \lceil 10/\varepsilon \rceil$ we obtain  $ q_i  \leq 1/m$. Taking $q=(q_1, \dots, q_d)$, it follows from Theorem \ref{teo:meeting} that $\lambda(q) \leq 10/m \leq \varepsilon$. Finally, by  (\ref{eq:main}) we have
$ \sup \E[W_n^2]<\infty$ and the lemma follows.
\end{proof}

\subsection{Bounding the probability of meeting}  \label{item1}

In this subsection we prove Item (a) in Theorem \ref{teo:meeting}. Observe that for all $d \geq m$ we have 
\[ \lambda(\underbrace{1/m, 1/m, \ldots, 1/m}_{m},\underbrace{ 0, \ldots, 0}_{d-m}) = \lambda(\underbrace{1/m, 1/m, \ldots, 1/m}_{m}).\]

For $m\leq 3$, we have that $\lambda(q^*)=\infty$, so we let $m \geq 4$. Then
\begin{align*}
\lambda(1/m, 1/m, \ldots, 1/m) & = \sum_{n=1}^{\infty} \sum_{\substack{(l_1, \ldots, l_m) \in \N^m \\ l_1+\cdots+l_m=n}} \binom{n}{l_1, \ldots, l_m}^2 \frac{1}{m^{2n}}   \\
&\leq \sum_{n=1}^{\infty} \left[ \max_{\substack{(l_1, \ldots, l_m) \in \N^m \\  l_1+\cdots+l_m=n} } \left\{  \binom{n}{l_1, \ldots, l_m}  \right\} \frac{1}{m^n} \right].
\end{align*}

We will split the first sum in two parts, the first for $n \leq m$ and the second for $n>m$, and bound each one separately.

Observe that for $n=1, \ldots, m$ the maximum inside the brackets is bounded by $n!$, so the sum for $n \leq m$ is bounded by
\begin{equation}
\sum_{n=1}^{m} \frac{n!}{m^n} \leq \frac{1}{m} + \frac{2}{m^2} + \sum_{n=3}^{m} \frac{3!}{m^3} \leq \frac{1}{m} + \frac{8}{m^2}.
\end{equation}

Now, for each $j \geq 1$ and $0 \leq \ell \leq m-1$, we use Stirling's bounds, 

\[ \sqrt{2 \pi n} \left( \dfrac{n}{e} \right)^n \leq n! \leq \sqrt{2 \pi n} \left( \dfrac{n}{e} \right)^n e^{\frac{1}{12n}},\] 

to obtain, for $n = jm+\ell$,

	\begin{align*}
\max_{\substack{(l_1, \ldots, l_m) \in \N^n \\  l_1+\cdots+l_m=n} } \left\{  \binom{n}{l_1, \ldots, l_m}  \right\}  &=  \frac{(jm+\ell)!}{[(j+1)!]^{\ell}(j!)^{m-\ell}} \\
&\leq \frac{ \sqrt{2 \pi (jm+ \ell)} \left( \dfrac{jm+ \ell}{e} \right)^{jm+ \ell} e^{\frac{1}{12(jm+\ell)}}}{ \left[ \sqrt{2 \pi (j+1)} \left( \dfrac{j+1}{e} \right)^{j+1} \right]^{\ell} \times \left[ \sqrt{2 \pi j} \left( \dfrac{j}{e} \right)^j \right]^{m-\ell}}\\
  &\leq \frac{e^{\frac{1}{12n}} \sqrt{m} \cdot  m^{jm+ \ell}}{\left[ \sqrt{2 \pi j}\right]^{m-1}}     \frac{  \left( j+ \frac{\ell}{m} \right)^{jm+ \ell}   }  {  \left( j+1 \right)^{j\ell+\ell} \times   j^{jm-j\ell}  }.
\end{align*}
Now, observe that
\begin{align*}
     \frac{  \left( j+ \frac{\ell}{m} \right)^{jm+ \ell}   }  {  \left( j+1 \right)^{j\ell+\ell} \times   j^{jm-j\ell}  }
     &=  \left( 1- \frac{m-\ell}{m(j+1)} \right)^{(j+1) \ell} \left( 1 + \frac{\ell}{mj} \right)^{j(m-\ell)}\\
     &\leq \exp \left(-\frac{(m-\ell)}{m} \ell \right) \times \exp\left( \frac{\ell}{m} (m - \ell) \right) \\
     &= 1,
\end{align*}
thus
\begin{equation}
\max_{\substack{(l_1, \ldots, l_m) \in \N^n \\  l_1+\cdots+l_m=n} } \left\{  \binom{n}{l_1, \ldots, l_m}  \right\} \leq  \dfrac{e^{\frac{1}{12 m}} \cdot \sqrt{m} \cdot m^{mj+\ell}}{(\sqrt{2 \pi})^{m-1} (\sqrt{j})^{m-1}}.
\end{equation}

Using the bound above we obtain
\begin{align*}
\lambda(q^*) & \leq \frac{1}{m} + \frac{8}{m^2} + \dfrac{e^{\frac{1}{12m}} \cdot m \cdot \sqrt{m}}{(\sqrt{2 \pi})^{m-1} }  \sum_{j=1}^{\infty} \frac{1}{j^{\frac{m-1}{2}}}   \\
&\leq \frac{1}{m} + \frac{8}{m^2} + \dfrac{e^{\frac{1}{12m}} \cdot m \cdot \sqrt{m}}{(\sqrt{2 \pi})^{m-1} }   \left( 1 + \frac{2}{m-3} \right).
\end{align*}

For $m \geq 4$ we then have
\begin{align} \label{equation:cotastirling}
\frac{8}{m^2} \leq \frac{2}{m} \ \  \ \ \mbox{and} \ \   \ \  \dfrac{e^{\frac{1}{12m}} \cdot m \cdot \sqrt{m}}{(\sqrt{2 \pi})^{m-1} }   \left( 1 + \frac{2}{m-3} \right) \leq \frac{7}{m},
\end{align}
and the result follows.

\begin{remark}
For any $\delta > 0$, both upper bounds of (\ref{equation:cotastirling}) could be taken as $\delta/m$ as long as $m$ is chosen to be sufficiently large. In this case, we could conclude that $\lambda(q^*)\leq (1+2\delta)/m$ in (a) of Theorem \ref{teo:meeting}.
\end{remark}

\subsection{Projecting the random walks}   \label{item2}

We now want to understand the behavior of a random walk in $\Z^d$. To do that we will split the random walk in $\Z^d$ in two "normalized" projections in $\Z^2$ and $\Z^{d-2}$ and consider the behavior of the two parts to determine the behavior of the main random walk.

Given two independent oriented $d$-dimensional random walks $\{S^1_n(q)\}_n$, $\{S^2_n(q)\}_n$ associated with the probability vector $q=(q_1, \ldots, q_d)$, for $i=1,2$, let $\{R^i_n(q)\}_n$  be two independent 2-dimensional oriented random walks associated with $(q_1,q_2)$ and  let $\{U^i_n(q)\}_n$ be two independent $(d-2)$-dimensional oriented random walks associated with $(q_3,\dots, q_d)$. Writing, for $i=1,2$, $S^i_n(q) = (S^i_n(q)_1, \ldots, S^i_n(q)_d)$, we define two complementary bi-dimensional oriented new random walks, $\{\tilde{S}^1_n(q)\}_n$ and $\{\tilde{S}^2_n(q)\}_n$ coupled with $\{S^1_n(q)\}_n$ and $\{S^2_n(q)\}_n$ respectively, where 
\[ \tilde{S}^i_n(q) = \left(S^i_n(q)_1 + S^i_n(q)_2, S^i_n(q)_3 + \cdots + S^i_n(q)_d \right).
\]
Clearly $\{\tilde{S}^1_n(q)\}_n$ and $\{\tilde{S}^2_n(q)\}_n$ are independents and have the same distribution of a random walk associated with the probability vector $\tilde{q} := (q_1 + q_2, q_3 + \cdots + q_d)$. We will omit the dependency on $q$ until the proof of Theorem \ref{teo:meeting}.






One can think of the those newly defined random walks defined above as psudo-projections of the original random walks and the next lemma will express the meeting probability of the first in term of the latter.

\begin{lemma} \label{lemma:passeio-projecao}
Let $\{S_n^1\}_n$ and $\{S_n^2\}_n$ be two random walks associated with the probability vector $q=(q_1, \ldots, q_d)$. Then
\begin{equation} \label{equacao:encontro-passeio}
 \Q(S^1_n=S^2_n) =   \sum_{\substack {(j,k) \in \N^2 \\ j+k=n}} \Q(\tilde{S}^1_n = \tilde{S}^2_n = (j,k))   \Q(R^1_j = R^2_j)  \Q(U^1_k = U^2_k).
\end{equation}

\end{lemma}


\begin{proof}
In fact
\begin{equation}
\Q(S^1_n=S^2_n) = \sum_{\substack{(j,k) \in \N^2 \\ j+k=n}} \Q(S^1_n=S ^2_n, \tilde{S}^1_n=\tilde{S}^2_n =(j,k)).
\end{equation}
Now, observe that fixed $(j,k) \in \N^2$ such that $j+k = n$, we have
\begin{align*}
\Q&(S^1_n=S^2_n, \tilde{S}^1_n=\tilde{S}^2_n =(j,k)) \\
= &\left[  \binom{j+k}{j}^2 (q_1+q_2)^{2j} \cdot (q_3+\cdots + q_d)^{2k}  \right] \cdot \left[  \sum_{l=0}^j \binom{j}{l}^2 \left(  \frac{q_1}{q_1+q_2}  \right)^{2l} \left(  \frac{q_2}{q_1+q_2}  \right)^{2(j-l)}   \right] \\
&\cdot \left[  \sum_{ l_3 + \cdots + l_d = k } \binom{k}{l_3, \ldots, l_d}^2  \left( \frac{q_3}{q_3+ \cdots + q_d} \right)^{2l_3}  \cdots  \left( \frac{q_d}{q_3+ \cdots + q_d} \right)^{2l_d} \right] \\
=  & \, \Q(\tilde{S}^1_n = \tilde{S}^2_n = (j,k)) \cdot  \Q(R^1_j = R^2_j) \cdot \Q(U^1_k = U^2_k).
\end{align*}
\end{proof}

Next, we state and prove an elementary lemma which will be useful to bound the second term in  (\ref{equacao:encontro-passeio}).

\begin{lemma} \label{lemma:f-increasing}
For each $x \in [0,1]$ let $\{Z_n(x)\}_n$ be a random walk over $\Z$ 
\[ Z_n(x) = \zeta_1(x) + \cdots + \zeta_n(x), \]
where $\{\zeta_i\}_{i \in \N}$ are i.i.d. random variables with \[\Q(\zeta_i(x) =0) = x\]
and
 \[\Q(\zeta_i(x) =-1)=\Q(\zeta_i(x)=1)=\dfrac{1-x}{2} . \]
Then, for each $n \in \N$ fixed, the function $F_n:[1/2,1] \to [0,1]$ given by
\begin{equation} \label{def:fnx}
F_n(x) = \Q(Z_n(x)=0).
\end{equation} 
is increasing .
\end{lemma}

\begin{proof}
We want to prove that for $1/2 \leq x \leq y \leq 1$ we have $F_n(y)-F_n(x) \geq 0$.
To do that we will write $F$ as a sum and analyze the terms of the sum separately.
Define
\begin{equation} \label{def: Y}
Y_n(x) := \#\{i \in [n]:\zeta_i(x)=0\},
\end{equation} 
so we can write
\[F_n(x) = \sum_{j=0}^n \Q(Z_n(x)=0|Y_n(x)=j)\Q(Y_n(x)=j).\]

Let us now analyze each part of the sum. We first observe that
\begin{displaymath}
a_{n-j} := \Q(Z_n(x)=0|Y_n(x)=j) = \left\{ \begin{array}{ll}
\displaystyle 0 & \textrm{if $n-j \equiv 1 \mod 2$}\\
  \binom{n-j}{{(n-j)}/{2}} \times \dfrac{1}{2^{n-j}} & \textrm{if $n-j \equiv 0 \mod 2$}
\end{array}\right. .
\end{displaymath}
Observe also that $a_{n-j}$ is well defined because it does not depend on $n$ or $j$, but only on  $n-j$.
It is also easy to see that for any $0 \leq k \leq \dfrac{n}{2} - 1$ we have $ a_{2k} > a_{2k+2}$ and thus $a_{2k} \geq a_{2\ell}$ for all $0 \leq k \leq \ell \leq n/2$.

Now, let us analyze the behavior of the function $g_j$ given by
\begin{equation} \label{def:gjx}
g_j(x) := \Q(Y_n(x)=j) = \binom{n}{j} x^j(1-x)^{n-j}.
\end{equation}
Taking the derivative, we have 
\[ g'_j(x) = \binom{n}{j} x^j(1-x)^{n-j} \left(\frac{j}{x} - \dfrac{n-j}{1-x} \right), \]
so that, for $|y-x|$ sufficiently small and $y>x> 1/2$,  we have
\begin{align}
g_j(y)-g_j(x) < 0, &\ \ \mbox{if $j \leq nx$}, \label{equation:g-cresce} \\
g_j(y)-g_j(x) > 0, &\ \ \mbox{if $j > nx$}. \label{equation:g-decresce}
\end{align}

Let $N = \{ 0 \leq j \leq n : j \equiv n \mod 2\}$; then 
\[F_n(y) - F_n(x) = \sum_{j \in N} [g_j(y) - g_j(x)] a_{n-j},\]
and using the fact that $\{a_{2\ell}\}_{0 \leq \ell \leq n/2}$ is decreasing, and (\ref{equation:g-cresce}) and (\ref{equation:g-decresce}) we have
\begin{align*}
F_n(y) - F_n(x) &= \sum_{j \leq  nx , j \in N} [g_j(y) - g_j(x)] a_{n-j}  +  \sum_{ j > nx , j \in N} [g_j(y) - g_j(x)]  a_{n-j} \\
& \geq \left[ \sum_{j \leq  nx , j \in N} [g_j(y) - g_j(x)] + \sum_{ j >  nx , j \in N} [g_j(y) - g_j(x)] \right] a_{n-\lfloor nx \rfloor}\\
& \geq \left[ \sum_{ j \in N} g_j(y) - \sum_{ j \in N} g_j(x) \right] a_{n-\lfloor nx \rfloor}.
\end{align*}
Now note that $\displaystyle \sum_{ j \in N} g_j(y)$ can be obtained as the sum of the expansion of two binomials
\[ \displaystyle \sum_{ j \in N} g_j(y) = \frac{1}{2}[ (y-(1-y))^{n} + (y+(1-y))^{n}] = \frac{1+(2y-1)^{n}}{2}\]
and hence  $\displaystyle \sum_{ j \in N} g_j(x)$ is an increasing function in $[1/2,1]$. It follows that $F_n(y)-F_n(x) \geq 0$.
\end{proof}

\begin{proof}[Theorem 2] We want to show that, the value of $\lambda$  is maximal when the positive entries of the vector are packed in $m$ coordinates. Let 
\[\mathcal{A} := \left\{ q = (q_1, \dots, q_d): 0 \leq q_i \leq \frac{1}{m}, \; \forall i=1,\dots, d \: \mbox{and} \; \sum q_i =1 \right \},\]
and given $q \in \mathcal{A}$ let
\[B(q) := \#\Big\{ i \in [d]: q_i \notin \{0, 1/m\} \Big\}.\]

With that in mind we will define a packing algorithm $A: \mathcal{A} \to \mathcal{A} $ which increases the value of $\lambda$ in each step. 

Let $q\in \mathcal{A}$. If $B(q)=0$, then define $A(q)=q^*$. If $B(q)>0$, then necessarily $B(q) \geq 2$, and suppose without loss of generality that neither $q_1$ nor $q_2$ belongs to $\{0, 1/m \}$, so we define $A(q) = (q'_1, \ldots, q'_d)$ where $q'_i=q_i$ for all $i\geq 3$, and
\begin{flalign*}
\text{for } q_{1}+q_{2} \leq 1/m, &\text{ we let } q'_{1} = q_{1}+q_{2} \text{ and } q'_{2} = 0, \\
\text{for }q_{1}+q_{2} > 1/m, &\text{ we let }
q'_{1} = q_{1}+q_{2}-1/m \text{ and } q'_{2} = 1/m, .
\end{flalign*}

We claim that the this algorithm has the following properties

\begin{enumerate}
\item if $B(q)=0$ then $B(A(q)) = 0$;

\item if $B(q)>0$ then $B(A(q)) < B(q)$;

\item $\lambda(A(q)) \geq \lambda(q)$.
\end{enumerate}


Properties 1 and 2 follow from the definition and we proceed with the proof that Property 3 holds.

We now compare the probabilities of the random walks associated with $q$ and $A(q)$. By Lemma \ref{lemma:passeio-projecao} we have 
\begin{multline*}
\Q(S_n^1(q) = S_n^2(q)) = \\ \sum_{\substack{(j,k) \in \N^2 \\ j+k=n}} \Q(\tilde{S}_n^1(q)=\tilde{S}_n^2(q) = (j,k)) \Q(R_j^1(q) = R_j^2(q)) \Q(U_k^1(q) = U_k^2(q)). 
\end{multline*} 

Observe that $\Q(U_k^1(q) = U_k^2(q))$ depends only on the last $d-2$ coordinates of $q$, so that
\[\Q\Big(U_k^1(q) = U_k^2(q)\Big) = \Q\Big(U_k^1(A(q)) = U_k^2(A(q))\Big).\]

Analogously, $\Q(\tilde{S}_n^1(q)=\tilde{S}_n^2(q) = (j,k))$ depends only on $q_1+q_2$ and since this sum is invariant by $A$ we have
\[ \Q\Big(\tilde{S}_n^1(q)=\tilde{S}_n^2(q) = (j,k)\Big) = \Q\Big(\tilde{S}_n^1(A(q))=\tilde{S}_n^2(A(q)) = (j,k)\Big). \]

We will now prove that 
\begin{equation} \label{finelquality}
\Q\Big(R_j^1(q) = R_j^2(q)\Big) \leq \Q\Big(R_j^1(A(q)) = R_j^2(A(q))\Big),    
\end{equation}
and it will follow that
\[ \lambda(q) = \sum_{n=1}^{\infty}\Q\Big(S_n^1(q) = S_n^2(q)\Big) \leq \sum_{n=1}^{\infty}\Q\Big(S_n^1(A(q)) = S_n^2(A(q))\Big) = \lambda(A(q)) .\]

To prove  (\ref{finelquality}) we observe that for each $q \in \mathcal{A}$, although $R_n^1(q) - R_n^2(q)$ is defined as a random walk on $\Z^2$, its trace lies on the secondary diagonal of $\Z^2$, and it has the distribution of a lazy random walk  $\{Z_n(x)\}_n$ as defined in Lemma \ref{lemma:f-increasing} with  $x = \frac{q_1^2+q_2^2}{(q_1+q_2)^2} \geq 1/2$. Analogously $R_n^1(A(q)) - R_n^2(A(q))$ has the same distribution as $\{Z_n(x')\}_n$, with $x' = \frac{(q_1')^2+(q_2')^2}{(q_1'+q_2')^2} \geq x$. Now,  (\ref{finelquality}) follows from Lemma \ref{lemma:f-increasing}.

Finally, we observe that for all $q\in {\cal A}$ we have $A^d(q)=q^*$,  where $A^d$ is the $d$-th iterated of $A$. Hence, by Property 3, we get $\lambda(q^*) = \lambda(A^d(q)) \geq \lambda(q)$ and this finishes the proof.

\end{proof}

\section{Final comments}\label{final}

We do not know whether Condition 2 of Theorem \ref{main}, the upper bound on the probabilities $p_i$, is only a technical limitation or if a big discrepancy on the anisotropy prevents the system to behave as in mean-field conditions even in arbitrary large dimensions. In any case, the isotropic case shows that some bound on the probabilities $p_i$ must be required as we explain now.

One of the  results in \cite{CD}, states that the isotropic critical parameter satisfies $p_c(d) \geq 1/d+1/(2d^3)+ o(1/d^3)$. Let now each $p_i=1/d+1/(3d^3)$, and take $d_0$ so that $p_i<p_c(d_0).$ In this case $\epsilon=1/(3d_0^2)$ and $p_i=\sqrt{3\epsilon}(1+\epsilon)$.

A natural related  question is whether anisotropic non-oriented percolation has the same limiting critical surface behavior, i.e., under which conditions can we guarantee that the critical surfaces stay close to the isotropic critical parameters in high dimensions.

{\bf Acknowledgments} P.A.G. has been supported by CAPES, A.P. has been supported by a PNPD/CAPES grant and R.S. has been partially supported by Conselho Nacional de Desenvolvimento Cient\'\i fico e Tecnol\'{o}gico (CNPq)  and by FAPEMIG (Programa Pesquisador Mineiro), grant PPM 00600/16.

\end{document}